\documentclass[12pt]{amsart}
\usepackage[english]{babel}
\usepackage{amsmath,amssymb,amsfonts,amsthm,amsopn}
\usepackage{graphics}
\usepackage{comment}

\newcommand{\modsp}{modulation space}

\newtheorem{tm}{Theorem}[section]
\newtheorem{lemma}[tm]{Lemma}

\newtheorem{theorem}{Theorem}[section]
\newtheorem{corollary}[theorem]{Corollary}

\theoremstyle{definition}
\newtheorem{proposition}[theorem]{Proposition}

\newtheorem{definition}[theorem]{Definition}

\newcommand{\beqa}{\begin{eqnarray*}}
\newcommand{\eeqa}{\end{eqnarray*}}

\newcommand{\field}[1]{\mathbb{#1}}
\newcommand{\bR}{\field{R}}        
\def\la{\lambda}

 \def\cF{\mathcal{F}}              
 \def\cS{\mathcal{S}}

 \def\cC{\mathcal{C}}

 \def\cN{\mathcal{N}}

\def\a{\aleph}

\def\rd{\bR^d}

\def\lrd{L^2(\rd)}

\def\intrd{\int_{\rd}}

\def\R{\right)}

\def\<{\left<}
\def\>{\right>}

\def\inv{^{-1}}

\def\mv1{M_v^1}

\def\o{\xi}
\def\a{\alpha}

\def\R{\mathbb{R}}
\def\Ren{\mathbb{R}^d}

\def\sch{\mathcal{S}}

\def\Fur{\mathcal{F}}
\def\Tau{\mathcal{T}}

\def\f{\varphi}

\def\Sn2{S_{2}(L^{2}(\Ren))}
\def\S1{S_{1}(L^{2}(\Ren))}
\def\sig00{\sigma_{0,0}}

\def\la{\langle}
\def\ra{\rangle}

\newcommand{\fpq}{W(\mathcal{F}L^p,L^q)}

\begin{document}

\title[Strichartz Estimates]{Strichartz Estimates for the Schr\"odinger equation}

\author{Elena Cordero and Davide Zucco}
\address{Department of Mathematics,  University of Torino, Italy}
\email{elena.cordero@unito.it}
\address{Department of Mathematics,  University of Torino, Italy}
\email{davide.zucco@unito.it}

\date{}

\begin{abstract}
The objective of this paper is to report on recent progress on
Strichartz estimates for the Schr\"odinger equation and to present
the state-of-the-art. These estimates have been  obtained in
Lebesgue spaces, Sobolev spaces and, recently, in Wiener amalgam
and modulation spaces. We present and  compare the different
technicalities. Then, we illustrate   applications to
well-posedness.
\end{abstract}

\maketitle

\section{Introduction}
In11 this note, we focus on the Cauchy problem  for Schr\"odinger
equations. To begin with, the Cauchy problem for the free
Schr\"odinger equation reads as follows
\begin{equation}\label{cp}
\begin{cases}
i\partial_t u+\Delta u=0\\
u(0,x)=u_0(x),
\end{cases}
\end{equation}
with $t\in\R$ and $x\in\R^d,\, d\geq1$. In terms of the Fourier
transform, we can write the solution as follows
\begin{equation}\label{scflow}
u(t,x)=\left(e^{it\Delta}u_0\right)(x):=\int_{\R^d} e^{2\pi
ix\cdot\xi}e^{- 4 \pi^2 it|\xi|^2} \widehat{u_0}(\xi) d\xi,
\end{equation}
where the Fourier multiplier $e^{it\Delta}$ is known as
\emph{Schr\"odinger propagator}.
The corresponding inhomogeneous equation is
\begin{equation}\label{cp2}
\begin{cases}
i\partial_t u+\Delta u=F(t,x)\\
u(0,x)=u_0(x),
\end{cases}
\end{equation}
with $t>0$ and $x\in\R^d,\, d\geq1$. By Duhamel's principle and
\eqref{scflow}, the integral version of \eqref{cp2} has the form
\begin{equation}\label{eqschrnonomo}
u(t,x)= e^{it\Delta}u_0(\cdot) + \int_0^t e^{i(t-s)\Delta} F(s,\cdot) ds.
\end{equation}

The study of space-time integrability properties of the solution
to \eqref{scflow} and \eqref{eqschrnonomo} has been pursued by
many authors in the last thirty years. The matter of fact is given by the
Strichartz estimates, that have become a fundamental and amazing
tool for the study of PDE's. They have been studied in the
framework  of different function/distribution spaces,\hfill like\hfill
Lebesgue,\hfill Sobolev,\hfill Wiener\hfill amalgam\hfill  and\hfill modulation \\
spaces and have
found  applications to well-posedness and scattering theory for
nonlinear Schr\"odinger equations
\cite{berezinshubin,cordero3,Dancona05,
degosson,GinibreVelo92,Kato70,keel,strichartz,tao2,tao,Yajima87}.

In this paper we exhibit these problems. First, in Section 3, we
introduce  the dispersive estimates and show how can be carried
out for these different spaces. The classical $L^p$ dispersive estimates
read as follows
$$\|e^{it\Delta} u_0\|_{L^r_x}\lesssim |t|^{-d(\frac12-\frac1r)} \|u_0\|_{L^{r'}_x},\quad 2\leq r\leq\infty,\quad\frac1r+\frac1{r'}=1.
$$

Section 4 is devoted  to the study
of  Strichartz estimates.  The nature of these estimates is
highlighted and  the results among different kinds of spaces are
compared with each others. Historically, the $L^p$ spaces \cite{GinibreVelo92,Kato70,keel,tao,Yajima87}
were the first to be looked at. The celebrated
homogeneous Strichartz estimates  for the solution
$u(t,x)=\left(e^{it\Delta}u_0\right)(x)$ read
\begin{equation}\label{S1}
\|e^{it\Delta} u_0\|_{L^q_t L^r_x}\lesssim \|u_0\|_{L^2_x},
\end{equation}
for $q\geq2$, $r\geq 2$, with
$2/q + d/r = d/2$, $(q,r,d)\not=(2,\infty,2)$, i.e., for $(q,r)$ {\it Schr\"odinger admissible} (see Definition \ref{DefSchradm}). Here, as usual, we set
\[
\|F\|_{L^q_tL^r_x}=\left(\int\|F(t,\cdot)\|^q_{L^r_x}\,dt\right)^{1/q}.
\]

In the sequel, the estimates for
Sobolev spaces were essentially  derived from the Lebesgue ones.
Recently, several authors (\cite{benyi,
benyi3,cordero,cordero2,cordero3,baoxiang,baoxiang2}) have turned
their attention to fixed time and space-time estimates for the
Schr\"odinger propagator between spaces widely  used in
time-frequency analysis, known as Wiener amalgam spaces and
modulation spaces. The first appearance of amalgam spaces can be
traced to  Wiener in his development of the theory of generalized
harmonic analysis \cite{wiener,wiener2,wiener3} (see \cite{Heil03}
for more details). In this setting, Cordero and Nicola
\cite{cordero, cordero2, cordero3} have discovered that the
pattern to obtain dispersive and Strichartz estimates is similar
to that of Lebesgue spaces.  The main idea is to show that the
fundamental solution $K_t$ (see  \eqref{chirp0} below) lies in the
Wiener amalgam space $W(L^p,L^q)$ (see Section 2 for the
definition) which generalizes the classical  $L^p$ space and,
consequently, provides a different information  between the local
and global behavior of the solutions. Beside the similar
arguments,  we point out also some differences, mainly in proving
the sharpness of dispersive estimates and Strichartz estimates.
Indeed, dilation arguments in Wiener amalgam and modulation spaces
don't work as in the classical $L^p$ spaces.

 Modulation spaces were introduced by
Feitchinger in 1980 and then were also redefined by Wang
\cite{baoxiang} using  isometric decompositions. The two different
definitions allow to look at the problem in two different manners.
As a result,  in \cite{baoxiang2}, a beautiful use of
interpolation theory on modulation spaces  allows to combine the
estimates obtained by means of the  classical  definition
 in \cite{benyi, benyi3} and the isometric definition in \cite{baoxiang, baoxiang2}, to obtain  more
general  fixed time  estimates in this framework. In order to
control the growth of singularity at $t=0$, we usually have the
restriction $d(1/2-1/p)\leqslant 1$; cf. \cite{cazenave, keel}. By
using the isometric decomposition in the frequency space, as in
\cite{baoxiang,baoxiang2}, one can remove the singularity at $t=0$
and  preserve the decay at $t=\infty$ in certain modulation
spaces. 

The Strichartz estimates  can be applied, e.g., to the well-posedness of non-linear Schr\"odinger equations or of linear Schr\"odinger equations with time-dependent potentials. We shall show  examples  in the last Section 5.\par
\vskip0.3truecm

 \textbf{Notation.} We define $|x|^2=x\cdot x$, for
$x\in\Ren$, where $x\cdot y=xy$ is the inner
 product on $\Ren$. The space of smooth functions with compact support is denoted by $\cC_0^\infty(\rd)$, the Schwartz class
  by $\sch(\Ren)$, the space of tempered distributions by $\sch'(\Ren)$. The Fourier transform is normalized to be ${\hat{f}}(\o)=\Fur f(\o)=\int f(t)e^{-2\pi i t\o}dt$.  Translation and modulation operators ({\it time and frequency shifts}) are defined, respectively, by $$ T_xf(t)=f(t-x)\quad{\rm and}\quad M_{\o}f(t)= e^{2\pi i \o  t}f(t).$$ We have the formulas $(T_xf)\hat{} = M_{-x}{\hat
{f}}$, $(M_{\o}f)\hat{} =T_{\o}{\hat {f}}$, and $M_{\o}T_x=e^{2\pi i x\o}T_xM_{\o}$. The notation $A\lesssim B$ means $A\leq c
B$ for a suitable constant $c>0$, whereas $A \asymp B$ means $c\inv A \leq B \leq c A$, for some $c\geq 1$. The symbol $B_1 \hookrightarrow B_2$
denotes the continuous embedding of the linear space $B_1$ into $B_2$.

\section{Function spaces and preliminaries}
In this section we present the function/distribution spaces we
work with, and the properties used in our study.
\subsection{Lorentz spaces}
 (\cite{stein93,steinweiss}). We recall that the Lorentz space $L^{p,q}$ on $\mathbb{R}^d$ is defined as the space of tempered distributions $f$ such that
\[
\|f\|^\ast_{pq}=\left(\frac{q}{p}\int_0^\infty [t^{1/p} f^\ast (t)]^q\frac{dt}{t}\right)^{1/q}<\infty,
\]
when $1\leq p<\infty$, $1\leq q<\infty$, and
\[
\|f\|^\ast_{pq}=\sup_{t>0} t^{1/p}f^\ast(t)<\infty
\]
when $1\leq p\leq\infty$, $q=\infty$. Here, as usual, $\lambda(s)=|\{|f|>s\}|$ denotes the distribution function of $f$
and $f^\ast(t)=\inf\{s:\lambda(s)\leq t\}$.\par One has $L^{p,q_1}\hookrightarrow L^{p,q_2}$ if $q_1\leq q_2$, and $L^{p,p}=L^p$. Moreover, for $1<p<\infty$ and $1\leq q\leq\infty$, $L^{p,q}$ is a normed space and its norm $\|\cdot\|_{L^{p,q}}$ is equivalent to the above quasi-norm $\|\cdot\|^\ast_{pq}$. \par
The function $|x|^{-\alpha}$ lives in $L^{d/\alpha,\infty}$, $0<\a<d$ but observe that this function doesn't live in any $L^p,\, 1\leqslant p\leqslant \infty$.
We now recall the following classical Hardy-Littlewood-Sobolev fractional integration theorem (see e.g. \cite[Theorem 1, pag 119]{stein} and \cite{stein93}), which will be used in the sequel
\begin{proposition}
Let $d\geqslant 1$, $0<\alpha<d$ and $1<p<q<\infty$ such that
\begin{equation}\label{conv1}
\frac{1}{q}=\frac{1}{p}-\frac{d-\alpha}{d}.
\end{equation}
Then the following estimate
\begin{equation}\label{eqhlsom}
||\,|\cdot|^{-\alpha}*f||_q \lesssim\,||f||_p
\end{equation}
holds for all $f\in L^p(\R^d)$.
\end{proposition}

\bigskip
\vskip0.1truecm
 \textbf{Potential and Sobolev spaces.} For
$s\in\R$, we define the Fourier multipliers $\la \Delta \ra^s
f=\cF^{-1}((1+|\cdot|^2)^{s/2}\hat{f})$, and $| \Delta |^s
f=\cF^{-1}(|\cdot|^s\hat{f})$. Then, for $1\leq p\leq\infty$, the
potential space ~\cite{bergh-lofstrom} is defined by
$$W^{s,p}=  \{f\in\sch',\,  \la \Delta \ra^s
f\in L^p\}$$ with norm $\|f\|_{W^{s,p}}=\|\la \Delta \ra^s
f\|_{L^p}$. The homogeneous potential space ~\cite{bergh-lofstrom}
is defined by
$$\dot{W}^{s,p}=  \{f\in\sch',\,  | \Delta |^s
f\in L^p\}$$ with norm $\|f\|_{\dot{W}^{s,p}}=\|| \Delta |^s
f\|_{L^p}$.

For $p=2$ the previous spaces are called Sobolev spaces $H^s$
and homogeneous Sobolev spaces $\dot{H^s}$, respectively.

\subsection{Wiener amalgam spaces}
(\cite{feichtinger80,feichtinger83,feichtinger90,fournier-stewart85,Fei98}). Let $g \in \cC_0^\infty$ be a test function that satisfies $\|g\|_{L^2}=1$. We will refer to $g$ as a window function.  For $1\leq p\leq \infty$, recall the $\cF L^p$ spaces, defined by
$$
\cF L^p(\rd)=\{f\in\cS'(\rd)\,:\, \exists \,h\in L^p(\rd),\,\hat h=f\};
$$
they are Banach spaces equipped with the norm
\[
\| f\|_{\cF L^p}=\|h\|_{L^p},\quad\mbox{with} \,\hat h=f.
\]
In the same way, for $1< p< \infty$, $1\leq q\leq \infty$, the Banach spaces $\cF L^{p,q}$ are defined by 
$$
\cF L^{p,q}(\rd)=\{f\in\cS'(\rd)\,:\, \exists \,h\in L^{p,q}(\rd),\,\hat h=f\};
$$
equipped with the norm
\[
\| f\|_{\cF L^{p,q}}=\|h\|_{L^{p,q}},\quad\mbox{with} \,\hat h=f.
\]
Let $B$ one of the following Banach spaces: $L^p, \cF L^p$, $1\leq p\leq \infty$, $\cF L^{p,q}$, $1< p< \infty$, $1\leq q\leq \infty$, valued in a Banach space, or also spaces obtained from these by real or complex interpolation.
Let $C$ be the $L^p$ space, $1\leq p\leq\infty$, scalar-valued. For any given function $f$ which is locally in $B$ (i.e. $g f\in B$, $\forall g\in\cC_0^\infty$), we set $f_B(x)=\| fT_x g\|_B$.

The {\it Wiener amalgam space} $W(B,C)$ with local component $B$ and global component  $C$ is
defined as the space of all functions $f$ locally in $B$ such that $f_B\in C$. Endowed with the norm $\|f\|_{W(B,C)}=\|f_B\|_C$, $W(B,C)$ is a Banach space. Moreover, different choices of $g\in \cC_0^\infty$  generate the same space and yield equivalent norms.

If  $B=\Fur L^1$ (the Fourier algebra),  the space of admissible windows for the Wiener amalgam spaces $W(\Fur L^1,C)$ can be
enlarged to the so-called Feichtinger algebra $W(\Fur L^1,L^1)$. Recall  that the Schwartz class $\sch$ is dense in $W(\Fur L^1,L^1)$.\par We use the following definition of mixed Wiener amalgam norms. Given a measurable function $F$ of the two variables $(t,x)$ we set
\[
\|F\|_{W(L^{q_1},L^{q_2})_tW(\Fur L^{r_1},L^{r_2})_x}= \|
\|F(t,\cdot)\|_{W(\Fur L^{r_1},L^{r_2})_x}\|_{W(L^{q_1},L^{q_2})_t}.
\]
Observe  that \cite{cordero}
\[
\|F\|_{W(L^{q_1},L^{q_2})_tW(\Fur L^{r_1},L^{r_2})_x}= \|F\|_{W\left(L^{q_1}_t(W(\Fur L^{r_1}_x,L^{r_2}_x)),L^{q_2}_t\right)}.
\]
The following properties of Wiener amalgam spaces  will be frequently used in the sequel.
\begin{lemma}\label{WA}
  Let $B_i$, $C_i$, $i=1,2,3$, be Banach spaces  such that $W(B_i,C_i)$ are well   defined. Then,
  \begin{itemize}
  \item[(i)] \emph{Convolution.}
  If $B_1\ast B_2\hookrightarrow B_3$ and $C_1\ast
  C_2\hookrightarrow C_3$, we have
  \begin{equation}\label{conv0}
  W(B_1,C_1)\ast W(B_2,C_2)\hookrightarrow W(B_3,C_3).
  \end{equation}
   In particular, for every
$1\leq p, q\leq\infty$, we have
\begin{equation}\label{p2}
\|f\ast u\|_{\fpq}\leq\|f\|_{W(\Fur L^\infty,L^1)}\|u\|_{\fpq}.
\end{equation}
  \item[(ii)]\emph{Inclusions.} If $B_1 \hookrightarrow B_2$ and $C_1 \hookrightarrow C_2$,
   \begin{equation*}
   W(B_1,C_1)\hookrightarrow W(B_2,C_2).
  \end{equation*}
  \noindent Moreover, the inclusion of $B_1$ into $B_2$ need only hold ``locally'' and the inclusion of $C_1 $ into $C_2$  ``globally''.
   In particular, for $1\leq p_i,q_i\leq\infty$, $i=1,2$, we have
  \begin{equation}\label{lp}
  p_1\geq p_2\,\mbox{and}\,\, q_1\leq q_2\,\Longrightarrow W(L^{p_1},L^{q_1})\hookrightarrow
  W(L^{p_2},L^{q_2}).
  \end{equation}
  \item[(iii)]\emph{Complex interpolation.} For $0<\theta<1$, we
  have
\[
  [W(B_1,C_1),W(B_2,C_2)]_{[\theta]}=W\left([B_1,B_2]_{[\theta]},[C_1,C_2]_{[\theta]}\right),
  \]
if $C_1$ or $C_2$ has
absolutely continuous norm.
    \item[(iv)] \emph{Duality.}
    If $B',C'$ are the topological dual spaces of the Banach spaces $B,C$ respectively, and
    the space of test functions $\cC_0^\infty$ is dense in both $B$ and $C$, then
\begin{equation}\label{duality}
W(B,C)'=W(B',C').
\end{equation}
  \end{itemize}
  \end{lemma}

\noindent The proof of all these results can be found in
  (\cite{feichtinger80,feichtinger83,feichtinger90,Heil03}).
  
  Finally, let us recall the following lemma \cite[Lemma 6.1]{cordero3}, that will be used in the last Section 5.
  \begin{lemma}
 Let  $1\leq p,q,r\leq\infty$. If
\begin{equation}\label{ind}\frac1p+\frac1q=\frac1{r^\prime},\end{equation}
then
\begin{equation}\label{mconvm}
W(\cF L^{p^\prime},L^p)(\Ren)
\cdot W (\cF
L^{q\prime},L^q)(\Ren)\subset
W(\cF L^r,L^{r^\prime})(\Ren)
\end{equation}
with  norm inequality  $\| f
h \|_{ W(\cF
L^r,L^{r^\prime})}\lesssim
\|f\|_{W(\cF
L^{p^\prime},L^p)}\|h\|_{W
(\cF L^{q^\prime},L^q)}$.
\end{lemma}
  \par

\subsection{Modulation spaces}(\cite{Fei83,grochenig}).
Let $g\in\cS(\rd)$ be a non-zero window function and consider the
so-called short-time Fourier transform (STFT) $V_gf$ of a
function/tempered distribution $f$ with respect to the the window
$g$:
\[
V_g f(x,\o)=\la f, M_{\o}T_xg\ra =\int e^{-2\pi i \o y}f(y)\overline{g(y-x)}\,dy,
\]
i.e.,  the  Fourier transform $\cF$ applied to $f\overline{T_xg}$.\par For $s\in\R$, we consider the weight function $\la
x\ra^{s}=(1+|x|^2)^{s/2}, x\in\rd. $ If $1\leq p, q\leq\infty$, $s\in\R$, the {\it modulation space} $\mathcal{M}^{p,q}_{s}(\R^d)$ is defined as the closure of the Schwartz class with respect to the norm
\[
\|f\|_{\mathcal{M}_{s}^{p,q}}=\left(\intrd\left(\intrd |V_gf(x,\o)|^p dx\right)^{q/p}\la\o\ra^{sq} d\o\right)^{1/q}
\]
(with obvious modifications when $p=\infty$ or $q=\infty$).\par
Among the properties of \modsp s, we record that they are Banach
spaces whose  definition is independent of the choice of the
window $g\in\cS(\rd)$, $\mathcal{M}^{2,2}=L^2$,
$(\mathcal{M}_{s}^{p,q})'=\mathcal{M}_{-s}^{p',q'}$, whenever $p,q<\infty$.\par Another
definition of these spaces uses the unite-cube decomposition of
the frequency space, we address interested readers to
\cite{baoxiang}.\par

\par
Finally we recall the behaviour of modulation spaces with respect to complex interpolation (see \cite[Corollary 2.3]{feichtinger83}).
\begin{proposition}\label{cintm}
Let $1\leq p_1,p_2,q_1,q_2\leq\infty$, with $q_2<\infty$. If $T$ is a linear operator such that, for $i=1,2$,
\[
\|Tf\|_{M^{p_i,q_i}}\leq A_i\|f\|_{M^{p_i,q_i}}\quad \forall f\in M^{p_i,q_i},
\]
then
\[
\|Tf\|_{M^{p,q}}\leq CA_1^{1-\theta}A_2^\theta\|f\|_{M^{p,q}}\quad \forall f\in M^{p,q},
\]
where $1/p=(1-\theta)/p_1+\theta/p_2$, $1/q=(1-\theta)/q_1+\theta/q_2$, $0< \theta<1$ and $C$ is independent of $T$.
\end{proposition}

We observe that definition and properties  of modulation spaces
refer to the case  $p,q\geqslant 1$. For the quasi-Banach case
$0<p,q<1$ see, e.g., \cite{benyi3, baoxiang, baoxiang2}.

\subsection{$T^*T$ method}\cite{GinibreVelo92, GinibreVelo}
The $T^*T$ method is an abstract tool of Harmonic Analysis, discovered by Tomas in 1975. This method allows to know the continuity of a linear operator $T$ (and thus  of its adjoint $T^*$), simply by the boundedness of the composition  operator $T^*T$.

For any vector space $D$, we denote by $D_a^*$ its algebraic dual, by  $\mathcal{L}_a(D,X)$ the space of linear maps from $D$ to some other vector space $X$, and by $\left\langle \f,f\right\rangle_D$ the pairing between
$D_a^*$ and $D$ ($f\in D$, $\f\in D_a^*$), taken to be linear in $f$  and antilinear in $\f$.
\begin{lemma}\label{lemt*t}
Let $\mathcal{H}$ be a Hilbert space, $X$ a Banach space, $X^*$ the dual of $X$, and $D$ a vector space densely contained in $X$. Let $T\in \mathcal{L}_a(D,\mathcal{H})$ and $T^*\in \mathcal{L}_a(\mathcal{H},D_a^*)$ be its adjoint, defined by
\[
\left\langle T^*h,f\right\rangle_D=\left\langle h,Tf\right\rangle, \qquad\forall f\in D,\quad\forall h\in \mathcal{H},
\]
where $\left\langle ,\right\rangle$ is the  inner product in
$\mathcal{H}$ (antilinear in the first argument). Then the
following three conditions are equivalent.
\\
(1) There exists $a, 0\leqslant a<\infty$ such that for all $f\in D$
\begin{equation}\label{le1}
\|Tf\|_{\mathcal{H}}\leqslant a\|f\|_X;
\end{equation}
(2) Let $h\in \mathcal{H}$. Then $T^*h$ can be extended to a continuous linear functional on $X$, and there exists  $a, 0\leqslant a<\infty$, such that for all $h\in \mathcal{H}$
\begin{equation}\label{le2}
\|T^*h\|_{X^*}\leqslant a\|h\|_{\mathcal{H}}.
\end{equation}
(3) Let $f\in X$. Then $T^*Tf$ can be extended to a continuous linear functional on X, and there exists $a, 0\leqslant a<\infty$, such that for all $f\in D$,
\begin{equation}\label{le3}
\|T^*Tf\|_{X^*}\leqslant a^2\|f\|_X.
\end{equation}
The constant $a$ is the same in all the three cases. If  one of (all) those conditions is (are) satisfied,  the operators $T$
 and $T^*T$  extend by continuity to bounded operators from $X$ to $\mathcal{H}$ and from $X$ to $X^*$, respectively.
\begin{proof}
 From the fact that $D$ is densely contained in $X$, it follows that $X^*$ is a subspace of $D_a^*$.\\
$(1)\Rightarrow(2)$. Let $h\in \mathcal{H}$. Then, for all $f\in D$
\[
|\left\langle T^*h,f \right\rangle_D| = |\left\langle h,Tf\right\rangle| \leqslant ||h||_{\mathcal{H}}\, ||Tf||_{\mathcal{H}} \leqslant a\,||h||_{\mathcal{H}}\,||f||_X.
\]
$(2)\Rightarrow(1)$. Let $f\in D$. Then, for all $h\in \mathcal{H}$
\[
|\left\langle h,Tf \right\rangle| = |\left\langle T^*h,f\right\rangle_D| \leqslant ||T^*h||_{X^*}\, ||f||_X \leqslant a\,||h||_{\mathcal{H}}\,||f||_X
\]
Clearly $(1)$ and $(2)$ imply $(3)$, and therefore $(1)$ or $(2)$ implies $(3)$.\\
$(3)\Rightarrow(1)$. Let $f\in D$. Then
\[
||Tf||^2 = |\left\langle Tf,Tf \right\rangle| = |\left\langle T^*Tf,f\right\rangle_D| \leqslant ||T^*Tf||_{X^*}\, ||f||_X \leqslant a^2\,||f||_X^2.
\]
Since $D$ is a dense subspace of $X$, we see that $T$ can be extended to a bounded linear functional from $X$ to $\mathcal{H}$.
\end{proof}
\end{lemma}

The following corollary is extremely  useful.


\begin{corollary}\label{cort*t}
Let $\mathcal{H,\, D}$ and two triplets $(X_i,T_i, a_i),\,i=1,2$, satisfy the  conditions of Lemma
\ref{lemt*t}. Then for all choices of $i,\,j=1,2,$  $\mathcal{R}(T_i^*T_j)\subset X_i^*$ and for all $f\in D$,
\begin{equation}\label{123}
\|T_i^*T_j f\|_{X_i^*}\leqslant a_ia_j \|f\|_{X_j}.
\end{equation}
In particular, $T_i^*T_j$ extends by continuity to a bounded operator from $X_j$ to $X^*_i$, and \eqref{123} holds for all $f\in X_j$.
\end{corollary}

Ginibre and Velo \cite{GinibreVelo92} applied Lemma \ref{lemt*t} and Corollary \ref{cort*t} to the bounded operator $T: L^1(I,\mathcal{H})\rightarrow \mathcal{H}$, defined  by
\begin{equation}\label{eqdefT}
Tf=\int_I U(-t)f(t)dt,
\end{equation}
where $I$ is an interval of $\R$ (possibly $\R$ itself) and  $U$ a
unitary strongly continuous one parameter group in $\mathcal{H}$.
Then its adjoint $T^*$ is the operator
\[
T^*h(t)=U(t)h
\]
from $\mathcal{H}$ to $L^\infty(I,\mathcal{H})$, where the duality is defined by the scalar products in $\mathcal{H}$ and in $L^2(I,\mathcal{H})$, such that $T^*T$ is the bounded operator from $L^1(I,\mathcal{H})$ to $L^\infty(I,\mathcal{H})$ given by
\[
T^*Tf=\int_I U(t-t')f(t')dt'.
\]
Clearly the conditions of Lemma \ref{lemt*t} are satisfied with $X=L^1(I,\mathcal{H})$, the operator $T$ defined in \eqref{eqdefT},
the constant $a=1$, and $\mathcal{D}$ any dense subspace of $X$.

Let us introduce  the retarded operator $(T^*T)_R$, defined by
\[
(T^*T)_R f(t)=(U_R *_t f)(t)=\int_I U_R(t-t')f(t')dt'
\]
where $U_R(t)=\chi_{+}(t)U(t):=\chi_{[0,\infty)}(t)U(t)$.

We recall that a space $X$ of distributions in space-time is
 said to be \emph{time cut-off stable}  if the multiplication by the characteristic function $\chi_J$, of an interval $J$ in time, is a
 bounded operator in $X$ with norm uniformly bounded with respect to $J$. The spaces under our consideration are  of the type
 $X=L_t^q(I,Y)$,
 where $Y$ is a space of distribution in the space variable and for which that property obviously holds.

\begin{lemma}\label{lemmaoptronc}
Let $\mathcal{H}$ an Hilbert space, let $I$ be an interval of
$\R$, let $X\subset\mathcal{S}'(I\times \R^d)$ be a Banach space,
let $X$ be time cut-off stable, and let the conditions of Lemma
\ref{lemt*t}  hold for the operator $T$ defined in
 \eqref{eqdefT}. Then the operator $(T^*T)_R$ is (strictly speaking extends to) a bounded operator from $L_t^1(I,\mathcal{H})$ to $X^*$
 and from  $X$ to $L_t^\infty(I,\mathcal{H})$.
\begin{proof}
We recall the proof for sake of clarity.  It is enough to demonstrate the second property, from which the first one follows by duality.
 Let $f\in D$. Then, for each $t$
\begin{align*}
\|(T^*T)_R f(t)\|_\mathcal{H}&=\|T\chi_+ (t-\cdot)f\|_\mathcal{H}\leqslant a\sup_t\{\|\chi_+(t-\cdot)\|_{\mathcal{B}(X)}\}\|f\|_X\\
&\leq C a \|f\|_X,
\end{align*}
by the unitary of $U$, the estimate \eqref{le1} of Lemma \ref{lemt*t}, and the time cut-off stability of $X$.
\end{proof}
\end{lemma}

\section{Fixed time estimates}

In this section we study estimates for the solution $u(t,x)$ to
the Cauchy problem \eqref{cp}, for fixed $t$. Since multiplication
on the Fourier transform side intertwines with convolution on the
space side, formula \eqref{scflow} can be rewritten as
\begin{equation}\label{sol}
u(t,x)=(K_t\ast u_0)(x),
\end{equation}
where $K_t$ is the inverse Fourier transform of the multiplier
$e^{-4\pi^2 i t |\xi|^2}$, given by
\begin{equation}\label{chirp0}
K_t(x)=\frac{1}{(4\pi i t)^{d/2}}e^{i|x|^2/(4t)}.
\end{equation}

First, we establish the estimates for Lebesgue spaces. Since
$e^{it\Delta}$ is a unitary operator, we obtain the \emph{$L^2$
conservation law}
\begin{equation}\label{l2}
\|e^{it\Delta}u_0\|_{L^2(\bR^d)}=\|u_0\|_ {L^2(\bR^d)}.
\end{equation}
Furthermore, since $K_t\in L^\infty$ with $\|K_t\|_{\infty}\asymp
t^{-d/2}$, applying Young inequality to the fundamental solution
\eqref{sol}  we obtain the \emph{$L^1$ dispersive estimate}
\begin{equation}\label{disp}
\|e^{it\Delta}u_0\|_{L^\infty(\bR^d)}\lesssim
|t|^{-d/2}\|u_0\|_{L^1(\bR^d)}.
\end{equation}
This shows that if the initial data $u_0$ has a suitable
integrability in space, then the evolution will have a power-type
decay in time.  Using the Riesz-Thorin theorem (see, e.g.,
\cite{steinweiss}), we can interpolate \eqref{l2} and \eqref{disp}
  to obtain the important \emph{$L^p$ fixed time estimates}
\begin{equation}\label{displp}
\|e^{it\Delta}u_0\|_{L^{r}(\bR^d)}\lesssim
|t|^{-d\left(\frac{1}{2}-\frac{1}{r}\right)}\|u_0\|_{L^{r^\prime}(\bR^d)}
\end{equation}
for all $2\leqslant r\leqslant\infty$, with $1/r+1/r'=1$. These
estimates represent  the complete range of $L^p$ to $L^q$ fixed
time estimates available. In this setting,  the necessary
conditions are usually obtained by scaling conditions (see, for
example, \cite[Exercise 2.35]{tao}, and \cite{nicola} for the
interpretation in terms of Gaussian curvature of the
characteristic manifold). The following proposition (\cite[page
45]{zucco}) is an example of this technique  in the case $p=q'$.

\begin{proposition}
Let $1\leqslant r\leqslant\infty$ and $\alpha\in \R$ such that
\begin{equation}\label{dispnecess}
\|e^{it\Delta}u_0\|_{L^r(\R^d)}\leqslant
Ct^{\alpha}\|u_0\|_{L^{r'}(\R^d)},
\end{equation}
for all $u_0\in S(\R^d),\, t\neq 0$ and some $C$ independent of
$t$ and $u_0$. Then $\alpha=-d(\frac{1}{2}-\frac{1}{r}),\,
r'\leqslant r$ (and thus $2\leqslant r\leqslant \infty$).
\end{proposition}
\begin{proof}
We can rescale the initial data $u_0$ by a factor $\lambda$ and
use \eqref{dispnecess} for
\[
v(x):=u_0(\lambda x), \qquad \lambda>0,\quad u_0\in\cS(\rd).
\]
The corresponding solution with $v(x)$ as initial data is $
u(\lambda^2t,\lambda x) $, where $u(t,x)=e^{it\Delta}u_0$.
Therefore, by \eqref{dispnecess} and the scaling property
\[
\|f(\lambda \cdot)\|_r=\lambda^{-d/r}\|f(\cdot)\|_r
\]
one has
\[
\lambda^{-d/r}\|u(\lambda^2t,\cdot)\|_{L^{r}(\R^d)}\leqslant Ct^{\alpha}\lambda^{-d/r'}\|u_0\|_{L^{r'}(\R^d)},
\]
for all $\lambda>0,\, t\neq 0$ and $u_0\in S(\R^d)$. Choosing
$t=\lambda^{-2}$, we obtain
\[
\|u(1,\cdot)\|_{L^{r}(\R^d)}\leqslant
C\lambda^{-2\alpha-\frac{d}{r'}+\frac{d}{r}}\|u_0\|_{L^{r'}(\R^d)},
\]
 for all $\lambda>0$ and $u_0\in S(\R^d)$. Since $\|u(1,\cdot)\|_{L^{r}(\R^d)}$ and $\|u_0\|_{L^{r'}(\R^d)}$ are two positive constants,
 we have
\[
\begin{split}
\text{for }\,\,\lambda\to\infty,\quad &\quad-2\alpha-\frac{d}{r'}+\frac{d}{r}\geqslant 0,\\
\text{for }\,\,\lambda\to 0,\quad
&\quad-2\alpha-\frac{d}{r'}+\frac{d}{r}\leqslant 0
\end{split}
\]
and then we obtain the necessary condition for $\alpha$. Moreover,
since $e^{it\Delta}$ is invariant under translation, by
\cite[Theorem 1.1]{hormander} we obtain $r'\leqslant r$, i.e.,
$2\leqslant r\leqslant \infty$. By standard density argument we
attain the desired result.
\end{proof}

For $s\in\R$, consider the Fourier multiplier $\la \Delta\ra^s$,
defined by $\la \Delta\ra^s f=\cF^{-1}(\la\cdot\ra^s\hat{f})$.
Then, from \eqref{displp} and the commutativity property of
Fourier multipliers, one immediately  obtains the \emph{$W^{s,r}$
fixed time estimates}
\begin{equation}
\|e^{it\Delta}u_0\|_{W^{s,r}(\bR^d)}\lesssim
|t|^{-d\left(\frac{1}{2}-\frac{1}{r}\right)}\|u_0\|_{W^{s,r^\prime}(\bR^d)}
\end{equation}
for all $s\in \R,\, 2\leqslant r\leqslant\infty$, $1/r+1/r'=1$.
Finally, we note that  the conservation law \eqref{l2} can be
rephrased in this setting as the \emph{$H^s$ conservation law}
\begin{equation}\label{hs2}
\|e^{it\Delta}u_0\|_{H^s(\bR^d)}=\|u_0\|_ {H^s(\bR^d)}.
\end{equation}
The Schr\"odinger propagator does not preserve any $W^{s,r}$ norm
other than the $H^{s}$ norm.

Now, we focus on Wiener amalgam spaces. $K_t$ in
\eqref{chirp0} lives in $W(\cF L^1, L^\infty)\subset L^\infty$, see \cite{benyi,cordero,baoxiang}.
This is the finest Wiener amalgam space-norm for $K_t$ which,
consequently, gives the worst behavior in the time variable. It is
also possible to improve the latter, at the expense of a rougher
$x$-norm, see \cite{cordero3}.  Indeed, since $K_t\in W(\cF L^p,
L^\infty)$ with norm (see \cite[Corollary 3.1]{cordero3})
\begin{equation}\label{kernelnormp}
\|K_t\|_{W(\cF L^p, L^\infty)}\asymp |t|^{-d/p}(1+ t^2)^{(d/2)(1/p-1/2)},
\end{equation}
from the fundamental solution \eqref{sol} and the convolution relations for Wiener amalgam spaces
 in Lemma \ref{WA}(i), it turns out, for $2\leq q\leq\infty$, the \emph{$W(\cF L^p,L^q)$ dispersive estimates}
\begin{equation}\label{est1}
 \|e^{it\Delta}u_0\|_{W(\Fur L^{q^\prime}, L^\infty)}\lesssim |t|^{d(2/q-1)}(1+t^2)^{d(1/4-1/q)} \|u_0\|_{W(\Fur L^q,L^1)}.
\end{equation}
As well as for Lebesgue spaces, we can use complex interpolation
between the dispersive estimates \eqref{est1} and the $L^2$
conservation law ($L^2=W(\cF L^2,L^2)$) to obtain the following
\emph{$W(\cF L^p,L^q)$ fixed time estimates}, that combine
 \cite[Theorem 3.5]{cordero} and \cite[Theorem 3.3]{cordero3}.

\begin{theorem}\label{tfix}
For $2\leq q, r,s\leq \infty$ such that
\[
\frac1 s=\frac1 r+\frac 2 q\left(\frac12-\frac1r\right),
\]
we have
\begin{equation}\label{indrel}
\|e^{it\Delta}u_0\|_{W(\Fur L^{s^\prime},  L^r)}\lesssim |t|^{d\left(\frac2q-1\right) \left(1-\frac2r\right)}(1+t^2)^{d\left
(\frac{1}{4}  -\frac{1}{q}\right)\left(1-\frac{2}{r}\right)}
\|u_0\|_{W(\Fur L^{s},L^{r^\prime})}
\end{equation}
In particular, for $s=2$,
\begin{equation}\label{est2i}
\|e^{it\Delta}u_0\|_{W(L^{2}, L^r)}\lesssim (1+
t^2)^{-\frac{d}2\left(\frac{1}{2} -
\frac{1}{r}\right)}\|u_0\|_{W(L^{2},L^{r^\prime})},
\end{equation}
and, for $s=r$,
\begin{equation}\label{dispw}
\|e^{it\Delta}u_0\|_{W(\Fur L^{r^\prime},  L^r)}\lesssim (|t|^{-2}+|t|^{-1})^{d\left(\frac{1}{2}-\frac{1}{r}\right)} \|u_0\|_{W(\Fur L^{r},L^{r^\prime})}.
\end{equation}
\end{theorem}
\begin{proof}
Let us sketch the proof for the sake of readers. Estimate  \eqref{indrel} follow by complex interpolation between estimate \eqref{est1},
which corresponds to $r=\infty$, and \eqref{l2}, which corresponds to $r=2$.\\
Indeed, $L^2=W(\Fur L^2,L^2)=W( L^2,L^2)$. Using  Lemma
\ref{WA}(iii), with $\theta=2/r$ (observe that $0<2/r<1$), and
$1/{s^\prime}=(1-2/r)/q^{\prime}+(2/r)/2$, so that relation
\eqref{indrel} holds, we obtain
\begin{align*}
\left[W(\Fur L^{q^\prime}, L^\infty),W(\Fur L^2,L^2)\right]_{[\theta]}&=W\left([\Fur L^{q^\prime},\Fur L^2]_{[\theta]}, [L^\infty,L^2]_{[\theta]}\right)\\
&=W(\Fur L^{s^\prime}, L^{r})
\end{align*}
and
\begin{align*}
\left[W(\Fur L^q, L^1),W(\Fur L^2,L^2)\right]_{[\theta]}&=W\left([\Fur L^q,\Fur L^2]_{[\theta]}, [L^1,L^2]_{[\theta]}\right)\\
&=W(\Fur L^{s}, L^{r^\prime}).
\end{align*}
This yields the desired estimate \eqref{indrel}.
\end{proof}

Let us  compare the previous results with the classical $L^p$
estimates. For $2\leqslant r\leqslant \infty,\, \cF
L^{r'}\hookrightarrow L^r$, and the inclusion relations for Wiener
amalgam spaces (Lemma \ref{WA} (ii)) yield $W(\cF L^{r'},
L^r)\hookrightarrow W(L^r,L^r)=L^r$ and
$L^{r'}=W(L^{r'},L^{r'})\hookrightarrow W(\cF L^{r}, L^{r'})$.
Thereby the estimate \eqref{dispw} is an improvement of
\eqref{displp} for every fixed time $t\neq0$, and also uniformly
for $|t|>c>0$. Moreover, in  \cite{cordero3} Cordero and Nicola
 proved that the range $r\geq2$ in \eqref{dispw} is sharp, and
the same for the decay
$t^{-d\left(\frac{1}{2}-\frac{1}{r}\right)}$ at infinity and the
bound $t^{-2d\left(\frac{1}{2}-\frac{1}{r}\right)}$, when $t\to
0$.

Modulation spaces are new settings inherited by time-frequency
analysis where the fixed time estimates recently have been
studied, see \cite{benyi, benyi3, baoxiang, baoxiang2}. Here,
instead of using the representation of the solution $u(t,x)$ in
\eqref{sol}, the solution is written in the form of Fourier
multiplier $e^{it\Delta}u_0$  as in \eqref{scflow}, see
\cite{benyi, benyi3}. Indeed,  a sufficient condition for the
boundedness of a Fourier multiplier on modulation spaces is that
its symbol is in $W(\cF L^1,l^\infty)$ (\cite[Lemma 8]{benyi}).
Moreover,  the Schr\"odinger symbol $\sigma=e^{-it|\xi|^2}$ lives
 in $W(\cF L^1,l^\infty)$ and  its norm is
\[
\begin{split}
\|\sigma\|_{W(\cF L^1,l^\infty)}&= \sup_x \int_{\R^d} |V_g \sigma(x,\omega)| d\omega \\
&\asymp (1+t^2)^{-d/4}\int_{\R^d}
e^{-\frac{\pi}{t^2+1}|\omega|^2}d\omega \asymp (1+t^2)^{d/4},
\end{split}
\]
where $g(\xi)=e^{-\pi|\xi|^2}$. Then, by \cite[Lemma 2]{benyi3}
(also for $s=0$ \cite[Corollary 18]{benyi}) one has that
$e^{it\Delta}$ extends to a bounded operator on $\mathcal{M}_s^{p,q}$, i.e.,
the \emph{$\mathcal{M}_s^{p,q}$  fixed time estimates}
\begin{equation}\label{dispmpq1}
\|u(t,x)\|_{\mathcal{M}_s^{p,q}}\lesssim (1+|t|)^{d/2} \|u_0\|_{\mathcal{M}_s^{p,q}},
\end{equation}
for all $s\geqslant 0$ and $1\leqslant p,q \leqslant \infty$. In
particular, modulation space properties are preserved by the time
evolution of the Schr\"odinger equation, in strong contrast with
the case of Lebesgue spaces. Observe that \eqref{dispmpq1}, in the
case $s=0$, was also obtained using isometric decompositions in
\cite{baoxiang}. Later, Wang, Zaho, Guo in \cite{baoxiang2} obtain
the following   fixed time estimates
\begin{equation}\label{dispmpq2}
\|u(t,x)\|_{\mathcal{M}_s^{p,q}}\lesssim (1+|t|)^{-d(1/2-1/p)}
\|u_0\|_{\mathcal{M}_s^{p',q}},
\end{equation}
for all  $s\in\R,\, 2\leqslant p\leqslant \infty$ and $1\leqslant
q\leqslant \infty$. Comparing \eqref{displp} with \eqref{dispmpq1}
and \eqref{dispmpq2}, we see that the singularity at $t=0$
contained in \eqref{displp} has been removed in \eqref{dispmpq1}
and \eqref{dispmpq2} and the decay rate in \eqref{dispmpq2} when
$t=\infty$ is the same one as in \eqref{displp}. The
estimate \eqref{dispmpq2} also indicates that $e^{it\Delta}$ is
uniformly bounded on $\mathcal{M}^{2,q}$. The complex interpolation between
the case $p=2$ in \eqref{dispmpq2}, and $p=\infty$ in
\eqref{dispmpq1} yields
\begin{equation}\label{dispmpq}
\|u(t,x)\|_{\mathcal{M}_s^{p,q}}\lesssim (1+|t|)^{d(1/2-1/p)} \|u_0\|_{\mathcal{M}_s^{p,q}},
\end{equation}
for all $2\leqslant p\leqslant \infty,\,s\geqslant 0$. However, it
is still not clear whether  the growth order on time in the
right-hand side of \eqref{dispmpq} is optimal.

\section{Strichartz estimates}

In many applications, especially in the study of well-posedness of
PDE's, it is useful to have estimates for the solution both in
time and space variables. In this direction, the main result  is
represented by  the \emph{Strichartz estimates}.  First, let us
introduce the following definitions.

\begin{definition}\label{DefSchradm}
Following \cite{keel}, we say that the exponent pair $(q,r)$ is \emph{Schr\"odinger-admissible} if $d\geqslant 1$ and
\[
2\leqslant q,r\leqslant \infty,\quad \frac{1}{q}=\frac{d}{2}\Big(\frac{1}{2}-\frac{1}{r}\Big), \quad (q,r,d)\neq (2,\infty,2).
\]
\end{definition}

\begin{definition}
Following \cite{foschi}, we say that the exponent pair $(q,r)$ is \emph{Schr\"odinger-acceptable} if
\[
1\leqslant q< \infty,\quad 2\leqslant r\leqslant \infty,\quad \frac{1}{q}<d\Big(\frac{1}{2}-\frac{1}{r}\Big),
 \text{ or } \quad (q,r)= (\infty,2).
\]
\end{definition}

The original version  of Strichartz estimates in $L^p$ spaces, closely related to restriction problem of Fourier
 transform to surfaces, was elaborated  by Robert Strichartz \cite{strichartz} in 1977(who, in turn, had precursors in \cite{segal, tomas}). In 1995 a brilliant idea of Ginibre and Velo \cite{GinibreVelo} was the use of the  $T^*T$ Method (Lemma \ref{lemt*t})  to detach the couple $(q,r)$ from $(q',r')$ (see also \cite{Yajima87}). The  study of the endpoint case $(q,r)=(2,2d/(d-2))$ is treated in \cite{keel}, where Keel and Tao prove the estimate also for the endpoint when $d\geq 3$ (for $d=2$, the endpoint is $(q,r)=(2,\infty)$ and the estimate is false). We shall give a standard proof of the \emph{$L^p$ Stichartz estimates}  in the non-endpoint cases \cite{DAnc,zucco} (see also \cite{tao} where the following theorem is proved  using an abstract lemma, the \emph{Christ-Kiselev Lemma}, which is very useful in establishing retarded Strichartz estimates).

\begin{theorem}\label{teostrichartzlp}
For any Schr\"odinger-admissible couples $(q,r)$ and $(\tilde{q},\tilde{r})$ one has the homogeneous Strichartz estimates
\begin{equation}\label{strom}
\|e^{it\Delta}u_0\|_{L_t^q L_x^r(\R\times\R^{d})}\lesssim \|u_0\|_{L_x^2(\R^d)},
\end{equation}
the dual homogeneous Strichartz estimates
\begin{equation}\label{strdualom}
\Big\|\int_\R e^{-is\Delta}F(s,\cdot)\,ds\Big\|_{L_x^2(\R^d)}\lesssim \|F\|_{L_t^{\tilde{q}'}L_x^{\tilde{r}'}(\R\times\R^{d})},
\end{equation}
and the inhomogenous (retarded) Strichartz estimates
\begin{equation}\label{strnonom}
\Big\|\int_{s<t} e^{i(t-s)\Delta}F(s,\cdot)\,ds\Big\|_{L_t^q L_x^r(\R\times\R^{d})}\lesssim \|F\|_{L_t^{\tilde{q}'}L_x^{\tilde{r}'}(\R\times\R^{d})}.
\end{equation}
\begin{proof}
We shall only prove this theorem in the non-endpoint case, when
$q\neq 2$, addressing the interested reader to \cite{keel} for
the whole study. We use the \emph{$T^*T$ method} as follows. Let
$(q,r)$ be Schr\"odinger admissible and  consider the linear
operator $T :L_t^{1}L_x^{2}\longrightarrow  L_x^2$, defined as
\[
T(F)=\int_\R e^{-is\Delta} F(s,\cdot) ds.
\]
Its adjoint $T^* : L_x^2 \longrightarrow L_t^{\infty}L_x^{2}$ is
the Schr\"odinger propagator \eqref{scflow}
\[
T^*(u) = e^{it\Delta}u.
\]
Applying Minkowski's inequality, the fixed time estimate \eqref{displp} and \eqref{conv1}, we obtain the diagonal untruncated estimates
\[
\begin{split}
\Big\|\int_{\R} e^{i(t-s)\Delta} F(s,\cdot)\, ds\Big\|_{L_t^q L_x^r(\R\times\R^{d})}&\leqslant \Big\|\int_{\R}\| e^{i(t-s)\Delta} F(s,\cdot)\|_{L_x^r(\R^d)}\, ds\Big\|_{L_t^q (\R)}\\
 &\lesssim \Big\|\,\|F\|_{L_x^{r'}(\R^d)}*\frac{1}{|t|^{d(\frac{1}{2}-\frac{1}{r})}}\Big\|_{_{L_t^q (\R)}}\\
 &\lesssim \|F\|_{L_t^{q'}L_x^{r'}(\R\times\R^{d})},
\end{split}
\]
whenever $2<q,r\leqslant \infty$ are such that
$\frac{2}{q}+\frac{d}{r}=\frac{d}{2}$, and for any Schwartz
function $F\in\cS(\R\times\R^{d})$. Then, using Lemma \ref{lemt*t},
one obtains the homogeneous Strichartz estimates \eqref{strom} and
the corresponding dual homogeneous Strichartz estimates
\eqref{strdualom}. Corollary \ref{cort*t} applied to the previous
two estimates yields the non-diagonal untruncate estimates:
\[
\Big\| \int_\R e^{i(t-s)\Delta}F(s,\cdot)\,ds\Big\|_{L_t^q L_x^r(\R\times\R^{d})}\leqslant \|F\|_{L_t^{\tilde{q}'}L_x^{\tilde{r}'}(\R\times\R^{d})}.
\]
By untruncated diagonal estimates one  obtains the diagonal ones for the truncated operator, noting that
\[
\begin{split}
\Big\|\int_{-\infty}^t e^{i(t-s)\Delta}F(s,\cdot)\,ds\Big\|_{L_t^qL_x^r(\R\times\R^{d})}&\leqslant \Big\|\int_{-\infty}^{t} \|e^{i(t-s)\Delta}F(s,\cdot)\|_{L_x^r(\R^d)}\,ds\Big\|_{L_t^q(\R)}\\
&\leqslant \Big\|\int_{\R} \|e^{i(t-s)\Delta}F(s,\cdot)\|_{L_x^r(\R^d)}\,ds\Big\|_{L_t^q(\R)}\\
&\lesssim \|F\|_{L_t^{q'}L_x^{r'}(\R\times\R^{d})}.
\end{split}
\]
Moreover,  using Lemma \ref{lemmaoptronc}, with
$X=L^{q'}_tL^{r'}_x$ and the truncated operator
$(T^*T)_RF(t)=\int_0^t e^{i(t-s)\Delta}F(s)\,ds,$ one obtains
\begin{equation}\label{s1}
\Big\|\int_0^t e^{i(t-s)\Delta}F(s,\cdot)ds\Big\|_{L_t^{\infty}L_x^2(\R\times\R^{d})} \lesssim \|F\|_{L_t^{q'}L_x^{r'}},
\end{equation}
for all admissible pairs $(q,r)$.
Then, by complex interpolation between this estimate and  the  diagonal truncated ones   above one gets the non-diagonal truncate estimates
\eqref{strnonom}, for any couple $(q,r),\,(\tilde{q},\tilde{r})$ Schr\"odinger admissible.
\end{proof}
\end{theorem}

The estimates are known to fail  at the endpoint $(q,r,d)=(2,\infty,2)$, see \cite{montgomery}, where Smith constructed a counterexample using the
 Brownian motion, although the homogeneous estimates can be saved  assuming spherical symmetry \cite{mnno, stefanov, tao2}. The exponents in the
 homogeneous estimates are optimal (\cite[Exercise 2.42]{tao});  some additional estimates are instead  available in the inhomogeneous case
 (see, for example, \cite{oberlin}). Indeed, Kato \cite{kato} proved that inhomogeneous estimates \eqref{strnonom} hold true
 when the pairs $(q,r)$ and $(\tilde{q},\tilde{r})$ are Schr\"odinger acceptable and satisfy the scaling condition
  $1/q+1/\tilde{q}=d/2(1-1/r-1/\tilde{r})$ in the range $1/r,1/\tilde{r}>(d-2)/(2d)$. Afterwards,  for $d>2$, Foschi \cite{foschi} improved this result by looking for the optimal range of Lebesgue exponents for which inhomogeneous Strichartz estimates hold (results almost equivalent have recently obtained by Vilela \cite{vilela}). Actually, this range is larger than the one given by admissible exponents for homogeneous estimates, as was shown by
 the following result \cite[Proposition 24]{foschi}.
\begin{proposition}
If $v$ is the solution to \eqref{cp2}, with zero initial data and inhomogeneous term $F$ supported on $\R\times\R^d$, then we have the estimate
\begin{equation}
\|v\|_{L_t^q L_x^r(\R\times\R^d)}\lesssim \|F\|_{L_t^{\tilde{q}'}
L_x^{\tilde{r}'}(\R\times\R^d)}
\end{equation}
whenever $(q,r),(\tilde{q},\tilde{r})$ are Schr\"odinger acceptable pairs which satisfy the scaling condition
\[
\frac{1}{q}+\frac{1}{\tilde{q}}=\frac{d}{2}\Big(1-\frac{1}{r}-\frac{1}{\tilde{r}}\Big),
\]
and either the conditions
\[\frac{1}{q}+\frac{1}{\tilde{q}}<1,\,\quad\frac{d-2}{r}\leqslant \frac{d}{\tilde{r}},\,\quad \frac{d-2}{\tilde{r}}\leqslant \frac{d}{r}
\]
or the conditions
\[\frac{1}{q}+\frac{1}{\tilde{q}}=1,\,\quad\frac{d-2}{r}< \frac{d}{\tilde{r}},\,\quad\frac{d-2}{\tilde{r}}< \frac{d}{r},\,\quad\frac{1}{r}\leqslant\frac{1}{q},\quad \frac{1}{\tilde{r}}\leqslant\frac{1}{\tilde{q}}.
\]
\end{proposition}

For a discussion about the sharpness of this proposition we refer
to \cite{foschi}, where explicit counterexamples are constructed
to show the necessary conditions for inhomogeneous Strichartz
estimates.

Since the Schr\"odinger operator $e^{it\Delta}$ commutes with
Fourier multipliers like $|\Delta|^s$ or $\left\langle
\Delta\right\rangle^s$, it is easy to obtain  Strichartz estimates
for potential and Sobolev spaces. In particular, if $I$ is an
interval containing the origin and  $u:I\times\R^d\to\mathbb{C}$
is the solution to the inhomogeneous Schr\"odinger equation with
initial data $u_0\in \dot{H}_x^s(\R^d)$, given by the Duhamel
formula (\ref{eqschrnonomo}),  then, applying $|\Delta|^s$ to both
sides of the equation and using the estimate of Theorem
\ref{teostrichartzlp}, one obtains
\[
\|u\|_{L_t^q \dot{W}_x^{s,r}(I\times\R^d)}\lesssim \|u_0\|_{\dot{H}_x^s(\R^d)}+\|F\|_{L_t^{\tilde{q}'}\dot{W}_x^{s,\tilde{r}'}(I\times\R^d)}
\]
for all Schr\"odinger admissible couples  $(q,r)$ and
$(\tilde{q},\tilde{r})$.  In particular, if one considers the
homogeneous case (i.e. $F=0$), the Sobolev embedding
$\dot{W}_x^{s,r}\hookrightarrow L_x^{r_1},\, 0<s<d/2$ and
$1/r_1=1/r-s/d$, yields the \emph{$\dot{H}^s$ Stichartz estimates}
\[
\|u\|_{L_t^q
L_x^{r_1}(I\times\R^d)}\lesssim \|u_0\|_{\dot{H}_x^s(\R^d)},\quad\frac{2}{q}+\frac{d}{r_1}+s=\frac{d}{2}.
\] Since $s>0$ one has
\[
\frac{2}{q}+\frac{n}{r_1}<\frac{n}{2},
\]
hence, for any fixed value of $s$, the new Schr\"odinger
admissible
 couple $(q,r_1)$ lies on a parallel line below the corresponding case $s=0$.

Strichartz estimates in Wiener amalgam spaces enable us to control
the local regularity and decay at infinity of the solution
\emph{separately}. For comparison, the classical estimates
\eqref{strom} can be rephrased in terms of Wiener amalgam spaces
as follows:
\begin{equation}\label{S1W}
\|e^{it\Delta}u_0\|_{W(L^q,L^q)_t W(L^r,L^r)_x}\lesssim \|u_0\|_{L_x^2}.
\end{equation}
In this framework, Cordero and Nicola perform these estimates
mainly in two directions. First, in \cite{cordero}, for
$q\geqslant4$ they modify the classical estimate \eqref{S1W} by
(conveniently) moving local regularity from the time variable to
the space variable. Indeed, $\cF L^{r'}\subset L^r$ if $r\geqslant
2$, but the bound in \eqref{dispw} is worse than the one in
\eqref{displp}, as $t\to 0$; consequently one has
\begin{equation}\label{necessarie}
\|e^{it\Delta}u_0\|_{W(L^{q/2},L^q)_t W(\cF L^{r'},L^r)}\lesssim
\|u_0\|_{L_x^2},
\end{equation}
for $4< q\leqslant \infty,\,2\leqslant r\leqslant \infty$, with $(q,r)$ Schr\"odinger admissible.
 When $q=4$ the same estimate holds with the Lorentz space $L^{r',2}$ in place of $L^{r'}$.
 Dual homogeneous and retarded estimates hold as well. Thereby, the solution averages locally in time by the $L^{q/2}$ norm, which
 is rougher than the $L^q$ norm in \eqref{strom} or, equivalently, in \eqref{S1W}, but it displays an $\cF L^{r'}$ behavior
  locally in space, which is better
 than $L^{r}$. In \cite{cordero3} it is shown the sharpness of these Strichartz estimates, except for the
threshold $q\geq4$, which seems quite hard to obtain. Secondly, in
\cite{cordero3}, a converse approach is performed,  by showing
that it is possible to move local regularity in \eqref{strom} from
the space variable to the time variable. As a result, new
estimates involving the Wiener amalgam spaces $W(L^p,L^q)$, that
generalize \eqref{strom}, are obtained, i.e., the following
\cite[Theorem 1.1]{cordero3}.

\begin{theorem}\label{prima2i} Let $1\leq
q_1,r_1\leq\infty$, $2\leq q_2,r_2\leq\infty$ such that $r_1\leq
r_2$,
\begin{equation}\label{pri1}\frac{2}{q_1}+
\frac{d}{r_1}\geq\frac{d}{2},
\end{equation}
\begin{equation}\label{pri2}\frac{2}{q_2}+
\frac{d}{r_2}\leq\frac{d}{2},
\end{equation}
$(r_1,d)\not=(\infty,2)$,
$(r_2,d)\not=(\infty,2)$ and,
if $d\geq 3$, $r_1\leq
2d/(d-2)$. Assume the same for
$\tilde{q}_1,\tilde{q}_2,\tilde{r}_1,\tilde{r}_2$.
Then, we have the homogeneous
Strichartz estimates
\begin{equation}\label{hom2i}\|e^{it\Delta}
u_0\|_{W(L^{q_1},L^{q_2})_t
W(
L^{r_1},L^{r_2})_x}\lesssim
\|u_0\|_{L^2_x},
\end{equation}
the dual homogeneous Strichartz estimates
\begin{equation}\label{dh2i}
\|\int e^{-is\Delta}
F(s)\,ds\|_{L^2}\lesssim
\|F\|_{W(L^{\tilde{q}_1^\prime},L^{\tilde{q}_2^\prime})_t
W(
L^{\tilde{r}_1^\prime},L^{\tilde{r}_2^\prime})_x},
\end{equation}
and the retarded Strichartz estimates
\begin{equation}\label{ret2i}
\|\int_{s<t} e^{i(t-s)\Delta}
F(s)\,ds\|_{W(L^{q_1},L^{q_2})_t W( L^{r_1},L^{r_2})_x} \lesssim\|F\|_{W(L^{\tilde{q}_1^\prime},L^{\tilde{q}_2^\prime})_t
W(L^{\tilde{r}_1^\prime},L^{\tilde{r}_2^\prime})_x}.
\end{equation}
\end{theorem}
This outcome is achieved by first establishing the estimates for the
particular case $q_1=\tilde{q_1}=\infty,\,r_1=\tilde{r}_1=2$, and
then by complex interpolation with the classical ones
\eqref{strom}.

 Figure 1 illustrates the range of exponents for the
homogeneous estimates when $d\geq3$. Notice that, if $q_1\leq
q_2$, these estimates follow immediately from \eqref{S1W} and the
inclusion relations of Wiener amalgam spaces. So, the issue
consists in the cases $q_1>q_2$.  Since there are no relations
between the pairs $(q_1,r_1)$ and $(q_2,r_2)$ other than $r_1\leq
r_2$, these estimates tell us, in a sense, that the analysis of
the local regularity of the Schr\"odinger propagator is quite
independent of its decay at infinity.\par

         
  \vspace{1.2cm}

In \cite{cordero3} it is proved that, for $d\geq3$, all the
constraints on the range of exponents in Theorem \ref{prima2i} are
necessary, except for $r_1\leq r_2$, $r_1\leq 2d/(d-2)$, which is still left open. However,  the following weaker result holds
\cite[Proposition 5.3]{cordero3}:

\medskip \emph{Assume $r_1>r_2$ and $t\not=0$. Then the propagator $e^{it\Delta}$ does not map $W(L^{r'_1},L^{r'_2})$ continuously
into $W(L^{r_1},L^{r_2})$.}
\medskip

\noindent shows that the estimates \eqref{hom2i} for exponents
$r_1>r_2$, if true, cannot be obtained from fixed-time estimates
and orthogonality arguments. The arguments employed for the
necessary conditions differ from the classical setting of Lebesgue
spaces, because the general scaling consideration does not work
directly. Indeed, the known bounds for the norm of the dilation
operator $f(x)\longmapsto f(\lambda x)$ between  Wiener amalgam
spaces (\cite{sugimototomita,Toft04}), yield constraints which are
weaker than the desired ones. So, the necessary conditions are
obtained  considering  families of rescaled Gaussians as initial
data, for which the action of the operator $e^{it\Delta}$ and the
involved norms can be computed explicitly, see \cite{cordero3}.

We end up this section with recalling Stricharz estimates for
modulation spaces. The main result in this framework is due to
Wang and Hudzik \cite{baoxiang2}. They use the same arguments as in
Keel and Tao \cite{keel}, who point out that the ranges of
exponents $(q,r)$ in \eqref{displp} could most likely be not
optimal. In fact, Keel and Tao  show that if the semigroup
$e^{it\Delta}$ satisfies the estimate
\begin{equation}\label{dispconfronto}
\|e^{it\Delta}u_0\|_{L^p}\lesssim (1+|t|)^{-d(1/2-1/p)}\|u_0\|_{L^{p'}}
\end{equation}
then \eqref{strom}, \eqref{strdualom} and \eqref{strnonom} hold if
one substitutes $q$ and $\tilde{q}$ by any $\gamma\geq \max(q, 2)$
and $\tilde{\gamma}\geqslant\max(\tilde{q},2)$, respectively.
 Since the estimate \eqref{dispmpq} is similar to \eqref{dispconfronto}, they optimize \eqref{strom}, \eqref{strdualom}
  and \eqref{strnonom} in  the function spaces $M_s^{p,q}$ to cover the exponents $(\gamma,q)$ and $(\tilde{\gamma},\tilde{q})$
  satisfying $\gamma\geqslant \max(q, 2)$ and $\tilde{\gamma}\geqslant\max(\tilde{q}, 2)$. Since the precise formulation of these results requires
  the introduction of other function spaces, we refer interested readers to  \cite[Section
  3]{baoxiang2}.

\section{Applications}

We start by focusing  on the Cauchy problem for the nonlinear
Schr\"odinger equation (NLS)
\begin{equation}\label{nl}
\begin{cases}
i\partial_t u+\Delta u+N(u)=0\\
u(0,x)=u_0(x).
\end{cases}
\end{equation}
The nonlinearity $N$ considered will be either power-like \[
p_k(u)=\lambda|u|^{2k}u, \quad k\in\mathbb{N},\,\lambda\in \R
\]
or exponential-like
\[
e_\rho(u)=\lambda(e^{\rho|u|^2}-1)u,\quad\lambda,\rho\in \R.
\]
Both nonlinearities are smooth. The corresponding equations having
power-like nonlinearities $p_k$ are sometimes referred to as
algebraic nonlinear Schr\"odinger equations. The sign of the
coefficient $\lambda$ determines the defocusing, absent, or
focusing character of the nonlinearity.

We shall study the well-posedness of \eqref{nl}, in different
spaces. Recall that the problem \eqref{nl} is locally well-posed
in e.g. $H^s_x(\rd)$ if, for any $u_0^*\in H^s_x(\rd)$, there
exists a time $T>0$ and an open ball $B$ in $H^s_x(\rd)$
containing $u_0^*$ and a subset of $C^0_tH^s_x([T,T]\times \rd)$
such that for each $u_0\in B$ there exists a unique solution $u\in
X$ to the equation \eqref{nl} and the map $u_0\mapsto u$ is
continuous from $B$ (with the $H^s_x$ topology)  to $X$ (with the
$C^0_tH^s_x([T,T]\times \rd)$ topology).

A fundamental  tool in well-posedness theory
 is the \emph{contraction theorem}.
 Let us first work abstractly, viewing \eqref{nl} as an instance of the more general
\begin{equation}\label{eqforint}
u=u_{\text{lin}}+DN(u)
\end{equation}
where $u_{\text{lin}}:=e^{it\Delta}u_0$ is the linear solution,
$N$ is the nonlinearity  and $D$ is the Duhamel operator
\[
DF(t,x):=\int_0^te^{i(t-s)\Delta}F(s,\cdot)ds.
\]
The following abstract tool \cite[Proposition 1.38]{tao} then
allows us to find the desired contraction map.

\begin{proposition}[Abstract iteration argument]\label{propiter}
Let $\cN,\, \mathcal{\Tau}$ be two Banach spaces. Let $D:\cN\to \mathcal{\Tau}$ be a
bounded linear operator with the bound
\begin{equation}\label{eqlimlin}
||DF||_ {\mathcal{\Tau}}\leqslant C_0||F||_\cN
\end{equation}
for all $F\in\cN$ and some constant $C_0>0$, and let
$N:\cS\to\cN$, with $N(0)=0$, be a nonlinear operator which is
Lipschitz continuous and obeys the bounds
\begin{equation}\label{eqlimnonlin}
||N(u)-N(v)||_{\cN}\leqslant\frac{1}{2C_0}||u-v||_{\mathcal{\Tau}}
\end{equation}
for all $u,v$ in the ball
$B_\epsilon:=\{u\in\cS:||u||_{\mathcal{\Tau}}\leqslant\epsilon\}$, for some
$\epsilon>0$. Then, for all $u_\text{lin}\in B_{\epsilon/2}$,
there exists a unique solution $u\in B_\epsilon$ to the equation
(\ref{eqforint}), with Lipschitz map $u_\text{lin}\mapsto u$ with
 constant at most  2. That is, we have
\begin{equation}\label{eqlipnonlin}
||u||_{\mathcal{\Tau}}\leqslant 2||u_\text{lin}||_{\mathcal{\Tau}}
\end{equation}
\begin{proof}
Observe that for  $v=0$ the estimate (\ref{eqlimnonlin}) becomes
\begin{equation}\label{eqlimnonlindim}
||N(u)||_{\cN}\leqslant\frac{1}{2C_0}||u||_{\mathcal{\Tau}}
\end{equation}
(since  $N(0)=0$ by hypothesis). Then, fix $u_\text{lin}\in
B_{\epsilon/2}$, and consider the map
\[
\phi(u):=u_\text{lin}+DN(u).
\]
Using \eqref{eqlimlin} and (\ref{eqlimnonlindim}) one has
\[
||\phi(u)||_{\mathcal{\Tau}}=||u_\text{lin}+DN(u)||_{\mathcal{\Tau}}\leqslant
\frac{\epsilon}{2}+\frac{C_0}{2C_0}\epsilon=\epsilon
\]
for all $u\in B_\epsilon$, i.e., $\phi$ maps the ball $B_\epsilon$
into  $B_\epsilon$. Moreover, $\phi$ is a contraction on
$B_\epsilon$, indeed by (\ref{eqlimlin}) and (\ref{eqlimnonlin})
one has
\[
\begin{split}
||\phi(u)-\phi(v)||_{\mathcal{\Tau}} &= ||DN(u)-DN(v)||_{\mathcal{\Tau}}\leqslant C_0||N(u)-N(v)||_\cN\\
&\leqslant C_0\frac{1}{2C_0}||u-v||_{\mathcal{\Tau}}=\frac{1}{2}||u-v||_{\mathcal{\Tau}},
\end{split}
\]
for all $u,v\in B_\epsilon$. Then, the contraction theorem asserts
that there exists a unique  fixed point $u$ for $\phi$ and
moreover the map $u_\text{lin}\mapsto u$ is Lipschitz with
constant at most 2, that is  (\ref{eqlipnonlin}).
\end{proof}
\end{proposition}

Proposition \ref{propiter} is the main ingredient of the results in \cite{benyi, cordero2, cordero3, Dancona05,
tao, baoxiang2}.

First, consider the NLS  \eqref{nl} with $N=p_k$, with the initial
data $u_0$  in the Sobolev space $H_x^s(\R^d)$. To study this
Cauchy problem it is convenient to introduce a single space $S^s$
that recaptures  all  the Strichartz norms at a certain regularity
$H_x^s(\R^d)$ simultaneously. For sake of simplicity, we reduce to
the case $s=0$ which corresponds to the case  $L_x^2$, introducing
the \emph{Strichartz space} $S^0(I\times\rd)$, for any time
interval $I$, defined as the closure of Schwartz class $\cS$ with
respect to the norm
\[
||u||_{\cS^0(I\times\R^d)}:=\sup_{A}||u||_{L_t^q L_x^r(I\times\R^d)},
\]
where the set $A$ is given by  $A:=\{(\infty,2), (q,r)\}$, with
$(q,r)$ Schr\"odinger admissible. We define also the space
$N^0(I\times\R^d):=L_t^{q'}L_x^{r'}$. Then, using Proposition
\ref{propiter} and the  $L^p$ Strichartz estimates of Theorem
\ref{teostrichartzlp} one can prove the following
\cite[Proposition 3.15]{tao}

\begin{theorem}[$L_x^2$ subcritical solution]
Let $k$ be subcritical for $L_x^2$ (that is, $0<k<\frac{2}{d}$)
and let $\mu=\pm 1$. Then the NLS \eqref{nl} is locally well-posed
in $L_x^2$ in a subcritical sense. Indeed, for any $R>0$ there
exists a time $T>0$ such that for all $u_0$ in the ball
$B_R:=\{u_0\in L_x^2(\rd):||u_0||_{L_x^2}<R\}$ there exists a
unique solution $u$ in $L_x^2$ of \eqref{nl} in the space
$\cS^0([-T,T]\times\rd)\subset C_t^0 L_x^2([-T,T]\times\R^d)$.
Moreover, the map $u_0\mapsto u$, from $B_R$ to
$\cS^0([-T,T]\times\R^d)$, is Lipschitz continuous.
\end{theorem}

For results in the framework of modulation spaces we address to
\cite{benyi3,baoxiang, baoxiang2}.   In particular, we examine
\cite{benyi3}. The main result, obtained only with the $M^{p,q}_s$ dispersive estimates \eqref{dispmpq1}, is the following.
\begin{theorem}
Assume that $u_0\in M^{p,1}_s(\R^d)$ and $N\in\{p_k, e_\rho\}$.
Then, there exists $T=T(\|u_0\|_{M^{p,1}_s})$ such that \eqref{nl}
has a unique solution $u\in C^0M_s^{p,1}([0,T]\times\R^d)$.
Moreover, if $T<\infty$, then $\limsup_{t\to
T}\|u(t,\cdot)\|=\infty$.
\begin{proof}
The proof is simply an application of the abstract iteration
argument. Let us write it for the nonlinearity $N=p_k$. We choose
the spaces ${\mathcal{\Tau}}:= C^0M_s^{p,1}([0,T]\times\R),\,\cN:=M_s^{p,1}$,
and the Duhamel operator
\[
D:=\int_0^t e^{i(t-s)\Delta}\cdot ds.\]
Then, it is sufficient to prove \eqref{eqlimlin} and
\eqref{eqlimnonlin}  in this setting. Then, by the Minkowsky integral inequality, $M^{p,q}_s$ dispersive estimates \eqref{dispmpq1} and \cite[Corollary 2]{benyi3} one has 
\[
\begin{split}
\Big\|\int_0^t e^{i(t-\tau)\Delta}(p_k(u))(\tau) d\tau\Big\|_{M_s^{p,1}}& \leqslant \int_0^t \| e^{i(t-\tau)\Delta}(p_k(u))(\tau)\|_{M_s^{p,1}}
 d\tau \\
& \leqslant c_1TC_T \sup_{t\in[0,T]}\|p_k(u)(t)\|_{M_s^{p,1}}\\
& \leqslant c_1c_2C_T T \|u(t)\|_{M_s^{p,1}}^{2k+1}
\end{split}
\]
where $C_T=\sup_{t\in[0,T)}(1+|t|)^{d/2}$. Choosing $T>0$ such that $c_1c_2C_T T\leqslant C_0$, it follows \eqref{eqlimlin} and by
\[
p_k(u)(\tau)-p_k(v)(\tau)=\lambda(u-v)|u|^{2k}(\tau)+\lambda
v(|u|^{2k}-|v|^{2k})(\tau),
\]
it follows \eqref{eqlimnonlin}.
\end{proof}
\end{theorem}

For  Wiener amalgam spaces there are no results for the NLS. In
\cite{cordero3} there is a result  concerning linear Schr\"odinger
equations with time-dependent potentials. Indeed, in
\cite{cordero3} the well-posedness is proved in $L^2$ of the
following Cauchy problem, for all $d\geq 1$,
\begin{equation}\label{cpP}
\begin{cases}
i\partial_t u+\Delta u=V(t,x)u,\quad t\in [0,T]=I_T, \,\,x\in\rd,\\
u(0,x)=u_0(x),
\end{cases}
\end{equation}
and for the class of potentials
\begin{equation}\label{pot3}
V\in L^\alpha(I_T;W(\cF L^{p^\prime},
L^p)_x),\quad\frac1{\a}+\frac{d}{p}\leq1,\ 1\leq\alpha<\infty.\
{d}< p\leq\infty.
\end{equation}

\begin{theorem}\label{tepot}
Consider the class of
potentials \eqref{pot3}.
Then, for all $(q,r)$ such that $2/q+d/r=d/2$, $q>4,r\geq2$, the Cauchy problem \eqref{cpP} has a
 unique solution
\begin{flalign*}
\text{(i)\;} u\in\mathcal{C}(I_T;L^2(\bR))\cap &\, L^{q/2}(I_T; W(\cF L^{r^\prime}, L^r)), \,\text{ if }d=1;\\
\text{(ii)\;} u\in\mathcal{C}(I_T;\lrd) &\cap L^{q/2}(I_T; W(\cF L^{r^\prime}, L^r))\, \cap \\
 &\cap L^{2}(I_T; W(\cF L^{2d/(d+1),2}, L^{2d/(d-1)})), \text{if }d>1.
\end{flalign*}

\end{theorem}
\begin{proof}
It is enough to prove the
case $d=1$. Indeed, for
$d\geq 2$, condition
\eqref{pot3} implies $p>2$,
so that  $\cF L^{p^\prime}
\hookrightarrow L^p$ and the
inclusion relations of Wiener
amalgam spaces yield $W(\cF
L^{p^\prime},L^p)\hookrightarrow
W( L^{p},L^p)=L^p$. Hence our
class of potentials is a
subclass of those of
\cite[Theorem 6.1]{cordero},
for which the quoted theorem
provides the desired
result.\par

We now turn to the case
$d=1$. The proof follows the
ones of \cite[Theorem 1.1,
Remark 1.3]{Dancona05} and
\cite[Theorem 6.1]{cordero}
(see also \cite{Yajima87}).
\par First of all, since the
interval $I_T$ is bounded, by
H\"older's inequality and by taking $p$ large, we may assume
$1/\alpha+d/p=1$.\par
 We
choose a small time interval
$J=[0,\delta]$  and set, for
$q\geq2$, $q\not=4$, $r\geq
1$,
\[
Z_{q/2,r}=L^{q/2}(J;W(\cF
L^{r^\prime}, L^r)_x).
\]
Now,  fix  an admissible pair
$(q_0,r_0)$ with $r_0$
arbitrarily large (hence
$(1/q_0,1/r_0)$ is
arbitrarily close to
$(1/4,0)$) and set
$Z=\mathcal{C}(J;L^2)\cap
Z_{q_0/2,r_0}$, with the norm
$\|v\|_Z=\max\{\|v\|_{\mathcal{C}(J;L^2)},\|v\|_{Z_{q_0/2,r_0}}\}$.
We have $Z\subset Z_{q/2,r}$
for all  admissible pairs
$(q,r)$
 obtained
by interpolation between
$(\infty,2)$ and $(q_0,r_0)$.
Hence, by the arbitrary of
$(q_0,r_0)$ it suffices to
prove that $\Phi$ defines a
contraction in $Z$.\par
 Consider now
the integral formulation of
the Cauchy problem, namely
$u=\Phi(v)$, where
\[
\Phi(v)=e^{it\Delta}u_0+\int_0^t
e^{i(t-s)\Delta}V(s)v(s)\,ds.
\]
By the homogeneous and
retarded Strichartz estimates
in \cite[Theorems 1.1,
1.2]{cordero} the following
inequalities hold:
\begin{equation}\label{nan}
\|\Phi(v)\|_{Z_{q/2,r}}\leq
C_0
\|u_0\|_{L^2}+C_0\|Vv\|_{Z_{(\tilde{q}/2)',\tilde{r}'}},
\end{equation}
for all admissible pairs
$(q,r)$ and
$(\tilde{q},\tilde{r})$,
$q>4,\tilde{q}>4$.\par
Consider now the case $1\leq
\alpha< 2$. We choose
$((\tilde{q}/2)',\tilde{r})=(\alpha,2p/(p+2))$.
Since $v\in L^\infty(J; L^2)$, applying \eqref{mconvm} for
$q=2$ we get
$$\|Vv\|_{W(\cF L^{\tilde{r}},L^{\tilde{r}^\prime})}\lesssim \|V\|_{W(\cF
L^{p^\prime},L^{p})}\|v\|_{L^2},$$
whereas H\"older's Inequality
in the time-variable gives
$$\|Vv\|_{Z_{(\tilde{q}/2)',\tilde{r}'}}\lesssim \|V\|_{L^{\alpha}(J;W(\cF
L^{p^\prime},L^{p}))}\|v\|_{L^\infty(J;L^2)}.
$$
The estimate \eqref{nan} then
becomes
\begin{equation}\label{nan2}
\|\Phi(v)\|_{Z_{q/2,r}}\leq
C_0
\|u_0\|_{L^2}+C_0\|V\|_{L^{\alpha}(J;W(\cF
L^{p^\prime},L^{p}))}\|v\|_{L^\infty(J;L^2)}.
\end{equation}
By taking $(q,r)=(\infty,2)$
or    $(q,r)=(q_0,r_0)$ one
deduces that $\Phi:Z\to Z$
(the fact that $\Phi(u)$ is
{\it continuous} in $t$ when
valued in $L^2_x$ follows
from a classical limiting
argument \cite[Theorem 1.1,
Remark 1.3]{Dancona05}).
 Also,  if $J$ is small enough,
$C_0\|V\|_{L^\alpha_t
L^p_x}<1/2$, and $\Phi$ is a
contraction. This gives a
unique solution in $J$. By
iterating this argument a
finite number of times one
obtains a solution in
$[0,T]$.

 \par The case $2\leq \alpha<\infty$ is similar.
\end{proof}

This result generalizes \cite[Theorem 6.1]{cordero}, by
treating the one dimensional case as well and allowing the
potentials to belong to Wiener amalgam spaces with respect to the
space variable $x$. Other results on Schr\"odinger equations with
potentials in $L_t^p L_x^q$ can be found in \cite{Dancona05}.

\vskip0.5truecm

\end{document}